\documentclass[11pt,reqno,A4]{amsart}

\usepackage{enumerate,graphicx}
\usepackage{paralist,mathtools}
\usepackage{eucal,mathrsfs,yhmath}
\usepackage{amsfonts,amssymb,amsmath,amsthm}
\usepackage[english]{babel}
\usepackage[latin1]{inputenc}
\usepackage{dsfont}
\usepackage[figurename=Fig.]{caption}
\usepackage[labelsep = colon]{caption}
\usepackage{subcaption}
\usepackage{tabularx}
\usepackage{wrapfig}
\usepackage{sidecap}

\DeclareFontFamily{U}{matha}{\hyphenchar\font45}
\DeclareFontShape{U}{matha}{m}{n}{
      <5> <6> <7> <8> <9> <10> gen * matha
      <10.95> matha10 <12> <14.4> <17.28> <20.74> <24.88> matha12
      }{}
\DeclareSymbolFont{matha}{U}{matha}{m}{n}
\DeclareFontFamily{U}{mathx}{\hyphenchar\font45}
\DeclareFontShape{U}{mathx}{m}{n}{
      <5> <6> <7> <8> <9> <10>
      <10.95> <12> <14.4> <17.28> <20.74> <24.88>
      mathx10
      }{}
\DeclareSymbolFont{mathx}{U}{mathx}{m}{n}

\DeclareMathSymbol{\obot}         {2}{matha}{"6B}
\DeclareMathSymbol{\bigobot}       {1}{mathx}{"CB}

\newcommand{\figref}[1]{Fig.~\ref{#1}}

\renewcommand{\eqref}[1]{equation~(\ref{#1})}

\newcommand{\CC}{\mathbb C}

\newcommand{\reg}{\mathop{\rm reg}}
\newcommand{\interior}{\mathop{\rm int}}
\newcommand{\re}{\mathop{\rm Re}}

\newcommand{\arcsinh}{\mathop{\rm arcsinh}}

\numberwithin{equation}{section}  

\newtheoremstyle{style1}% name of the style to be used
  {5pt}% measure of space to leave above the theorem. E.g.: 3pt
  {5pt}% measure of space to leave below the theorem. E.g.: 3pt
  {\it}% name of font to use in the body of the theorem
  {}% measure of space to indent
  {\bf\scshape}% name of head font
  {.}% punctuation between head and body
  {5pt}% space after theorem head; " " = normal interword space
  {}% Manually specify head
\newtheoremstyle{style2}% name of the style to be used
  {3pt}% measure of space to leave above the theorem. E.g.: 3pt
  {3pt}% measure of space to leave below the theorem. E.g.: 3pt
  {}% name of font to use in the body of the theorem
  {}% measure of space to indent
  {\bf}% name of head font
  {:}% punctuation between head and body
  {3pt}% space after theorem head; " " = normal interword space
  {}% Manually specify head
\theoremstyle{style1}
\newtheorem{definition}{Definition}
\newtheorem*{definition*}{Definition}
\newtheorem{theorem}{Theorem}
\newtheorem{proposition}{Proposition}

\theoremstyle{style2}

\newtheorem*{example}{Example}

\usepackage[%from peter schroeder
  breaklinks, %make formatting of line broken links work
  colorlinks=true, %indicated clickable things in blue
  linkcolor=black, %removes all the gaudy colors that come with the previous choice
  anchorcolor=black,
  citecolor=black,
  filecolor=black,
  menucolor=black,
  urlcolor=blue,
  bookmarks,
  bookmarksnumbered,
  hyperfootnotes=false,
  pdfpagelabels,
  pdfstartview=FitH, %matter of taste what you want to force here
  pdfpagemode=UseOutlines, %Open with bookmarks side pane
  pdfnewwindow=true,
  pagebackref=false, %else formatting can be quite screwed up in bib
  bookmarksopen=false,
  bookmarks=true
]{hyperref}
\usepackage[figure]{hypcap}

\expandafter\def\expandafter\normalsize\expandafter{%
    \normalsize
    \setlength\abovedisplayskip{10pt}
    \setlength\belowdisplayskip{10pt}
    \setlength\abovedisplayshortskip{10pt}
    \setlength\belowdisplayshortskip{10pt}
}

\begin{document}

\title[A Weierstrass Representation for 2D Elasticity]{A Weierstrass Representation for 2D Elasticity}

\author{Ulrich Pinkall}
\author{Jonas Tervooren}

\address{Sekr. MA 8-1\\
  Technische Universit\"at Berlin\\Institut f\"ur Mathematik\\
  Stra{\ss}e des 17.\ Juni 136\\
  10623 Berlin\\ Germany}

\email{pinkall@math.tu-berlin.de, tervooren@math.tu-berlin.de}

\date{\today}

\begin{abstract}
We study a class of elastic energy functionals for maps between planar domains
(among them the so-called {\em squared distance functional}) whose critical points
({\em elastic maps}) allow a far more complete theory than one would expect
from general elasticity theory. For some of these functionals elastic maps even
admit a ``Weierstrass representation'' in terms of holomorphic functions,
reminiscent of the one for minimal surfaces. We also prove a global uniqueness
theorem that does not seem to be known in other situations.\\[0.5cm]
\end{abstract}

\thanks{Work supported by Berlin Mathematical School and SFB Transregio 109. We thank SideFX Software for Houdini licenses.}

\maketitle

\section{Introduction}
\label{intro}
The theory of elastic equilibrium goes all the way back to Bernoulli and Euler \cite{todhunter}. Its basic concern are the critical points of an elastic energy of the form
\[f\mapsto E(f)=\int_M W(f')\]
where $M\subset \mathbb{R}^n$ is a bounded domain and $f:M\to \mathbb{R}^n$ is a smooth map. If the material to be modelled is homogeneous and isotropic (i.e. at each point it has the same properties and no preferred direction) then the function $W:\mathbb{R}^{n\times n}\to \mathbb{R}\cup \infty$ is supposed to satisfy
\[W(RAS)=W(A).\]
for all $R,S\in \mbox{SO}(n)$. $W(f')$ describes the energetic response of the material to the failure of $f'$ to be an orientation-preserving orthogonal map. So $W(A)$ should be non-negative and assume its absolute minimum zero on $\mbox{SO}(n)$. Recent studies \cite{rigidity} have found it useful to make sure that this minimum is non-degenerate in the sense that there are a constants $\delta, C>0$ such that for all $A\in \mathbb{R}^{n\times n}$ with $d(A,\mbox{SO}(n))\leq\delta$ we have
\[W(A)\geq C\cdot d(A,\mbox{SO}(n))^2.\]
Here $d$ denotes the euclidean distance in the space of matrices endowed with the Frobenius norm. In fact, the choice
\[W_d(A):= \frac{1}{2} d(A,\mbox{SO}(n))^2\]
itself yields a valid elastic energy, called the {\em distance-squared energy} $E_d$, which has been successfully applied to elasticity simulations in the context of Computer Graphics \cite{chao}. The most classical choice (which yields the Saint Venant-Kirchhoff energy) is
\[W_{sv}(A):=|\!|A^tA-I|\!|^2.\]
In order to obtain existence and uniqueness results concerning minimizers of the elastic energy one usually has to specify suitable boundary conditions. So far such results only have been found for maps $f$ that are close to the the identity. Little is known about equilibria in the case of ``large deformations''. Moreover, only for very special elastic equilibria $f$ explicit formulas are available.

Let us compare this to another classical variational problem: Given a compact domain $M\subset \mathbb{R}^2$, look for smooth maps $f:M\to\mathbb{R}^3$ (subject to suitable boundary conditions) that are {\em minimal surfaces}, i.e. critical points for the area functional. To eliminate the freedom of repararametrization one usually restricts attention to conformal immersions $f$. Here the situation is quite different: There is an abundance of global classification results and {\em all} such $f$ can be explicitly expressed in terms of two holomorphic functions $g,h:M\to \mathbb{C}$:
\[f(z)=\mbox{Re} \int h(1-g^2,i(1+g^2),2g).\]
In this paper we will show that in two dimensions elastic equilibria based on the distance-squared energy admit a very similar Weierstrass representation in terms of two holomorphic functions. In fact, we will exhibit such a representation for a whole one-parameter family $E_\lambda$ of elastic energies that contains the distance-squared energy $E_d$ for $\lambda=1$.

Most of the theory we are going to develop applies to an even larger class of energies $E_V$ that depend on a certain convex function $V$ of one variable. We will prove the following global uniqueness result: Let $M\subset \mathbb{R}^2$ be a simply connected open domain and $f:M\to \mathbb{R}^2$ a critical point of $E_V$ (with respect to arbitrary variations) that is stable in a sense that we will make precise. Then $f$ is a rigid motion.

There are counterexamples if the dimension of $M$ is at least three or if $M$ is not simply connected.

\section{Elastic energies in the planar case} 
 \label{elastic energy}

The euclidean space of all real $2 \times 2$ matrices splits
as an orthogonal direct sum
\begin{equation}
gl(2,\mathbb{R}) = \mathbb{C} \oplus \mathbb{C}^\perp. \label{splitting}
\end{equation}
Here $\mathbb{C}$ consists of the orientation preserving conformal
endomorphisms of $\mathbb{R}^2$ and $\mathbb{C}^\perp$ of the orientation
reversing ones, i.e. elements of $\mathbb{C}$ are complex linear and elements of $\mathbb{C}^\perp$ are complex anti-linear.
In this notation $SO(2)$ is just the unit circle in $\mathbb{C}$.
For a smooth map $f: M \to \mathbb{R}^2$, the splitting (\ref{splitting})
reflects in the decomposition of the differential $df$ (viewed as an
$\mathbb{R}^2$-valued 1-form) as
\begin{equation}
\label{splitting df}
df = f_{z} dz + f_{\bar{z}} d\bar{z},
\end{equation}
where subscripts denote partial derivatives and
\begin{equation*}
f_{z} = (f_x - i f_y)/2 \quad  \textnormal{ and } \quad  f_{\bar{z}} = (f_x + i f_y)/2.
\end{equation*}
In this notation the volume form on $\mathbb{R}^2$ is
\begin{equation*}
dx\wedge dy = \frac{i}{2} dz \wedge d\bar{z}.
\end{equation*}
\begin{proposition}
The distance-squared energy of a smooth map $f:M\rightarrow \mathbb{R}^2$ is given by
\begin{equation}
E(f) = \frac{1}{2} \int_M (|f_{z}|-1)^2 + |f_{\bar{z}}|^2. \label{squaredDistance}
\end{equation}
\end{proposition}

\begin{proof}
\begin{align*}
E\left(f\right) &= \frac{1}{2} \int_M d(df,SO(2))^2  \\
&=\frac{1}{2} \int_M \min_{B\in SO(2)} \| \underbrace{f_zdz-B}_{\in \mathbb{C}} + \underbrace{f_{\bar{z}}d\bar{z}}_{\in \mathbb{C}^\perp} \|^2\\
&= \frac{1}{2} \int_M \min_{B\in SO(2)} \| f_zdz-B \|^2 + \|f_{\bar{z}}d\bar{z} \|^2\\
&= \frac{1}{2} \int_M \left(\left| f_z \right|-1\right)^2 + \left|f_{\bar{z}}\right|^2. 
\end{align*}
\end{proof}

We will also investigate a modified version of the distance-squared energy $E(f)$ given by
\begin{equation}
E_V(f) := \frac{1}{2}\int_M V(|f_{z}|^2) + |f_{\bar{z}}|^2, \label{EV}
\end{equation}
where $V:\left(0,\infty\right)\rightarrow \mathbb{R}$ is a smooth function with the following properties:
\begin{enumerate}
\item $V$ is strictly convex.
\item $V$ takes its only minimum zero at $x=1$.
\item $V'(x)\sqrt{x} \in O(1)$ near $x=0.$
\end{enumerate}
The simplest case just inserts a constant $\lambda > 0$ in $E(f)$:
\begin{equation}
E_\lambda(f) = \frac{1}{2} \int_M \lambda(|f_{z}|-1)^2 + |f_{\bar{z}}|^2. \label{Elambda}
\end{equation}
The crucial property of $E_V$ is the fact that one can add to $E_V$ one half of the
oriented area of $f(M)$ to obtain an expression that depends on $f_z$ only:
\begin{align}
E_V+ \frac{1}{2} \mbox{area}(f(M)) &= \frac{1}{2} \int_M V(|f_{z}|^2) + |f_{\bar{z}}|^2 + \det(df) \nonumber \\ 
&= \frac{1}{2} \int_M V(|f_{z}|^2) + |f_z|^2.
\label{E+A}
\end{align}
Here we have used
\begin{equation*}
\det(df)=|f_z|^2-|f_{\bar{z}}|^2.
\end{equation*}
Since the area of $f(M)$ is unaffected by variations of $f$ supported in the interior of $M$, this modification of $E_V$ neither changes the Euler-Lagrange equations nor the stability properties with respect to variations with fixed boundary. As a consequence, the property of being a critical point of $E_V$ is invariant under the addition of antiholomorphic functions.
% This is the reason why holomorphic functions make their
% appearance in the treatment of the critical points of $E_V$.

\section{Euler-Lagrange equations for $E_V$}
\label{EL}
Let $M \subset \mathbb{R}^2$ be a domain with piecewise smooth boundary. The class of maps $f: M \to \mathbb{R}^2$ we are most interested
in are smooth orientation preserving immersions. On the other hand, when pushed far enough from the resting state by the boundary conditions, we will see that elastic maps have the tendency to develop branch points (\figref{branchpoint}). Accordingly, we include a larger class of maps:

\begin{definition}
A map $f: M \to \mathbb{R}^2$ is called almost smooth if it is Lipschitz
and smooth away from finitely many points.
\end{definition}

Almost smooth maps form a vector space. Moreover, since the derivative of an
almost smooth map $f$ is bounded, $E_V(f)$ is well-defined. As a consequence of (\ref{E+A}) we saw that adding to $f$ an antiholomorphic map preserves the Euler-Lagrange equations for the
energy $E_V$. Asking our maps to be orientation preserving immersions away from
finitely many points would break this natural symmetry. We therefore weaken this
condition in a way that effectively says that (away from finitely many points)
locally $f$ becomes an orientation preserving immersion after adding a suitable
antiholomorphic function:

\begin{definition}
A Lipschitz map $f: M \to \mathbb{R}^2$ is called almost immersed if on the
complement of finitely many points $\{p_1, \ldots ,p_n\}$ (called the regular
part $\reg(f)$ of $f$) it is smooth and $f_z$ does not vanish.
\end{definition}

Indeed, for every point $p \in \reg(f)$ there exists an open neighbourhood $U \subset \reg(f)$ on which we can define the antiholomorphic function $k:U\rightarrow\mathbb{R}^2$,
\begin{equation}
\label{k}
 k(z):= -\bar{z} f_{\bar{z}}(p),
\end{equation}
such that, $\tilde{f}:= f + k$ is an orientation preserving immersion on $U$.\\

To derive the Euler-Lagrange equations for $E_V$ our strategy is as follows: First we
obtain a necessary condition for being a critical point by allowing only a
restricted class of variations that do not move the branch points. We will see that maps satisfying this necessary condition fall into two categories: The first consists of global minima among almost immersed maps with the same boundary values. So in particular maps in this category are certainly honest critical points. The second consists of maps $f$ that are unstable even locally. This means that our Euler-Lagrange equation (even though derived based on a restricted class of variations) captures at least all the local minima of the energy.

\begin{definition}
An almost immersed map $f: M \to \mathbb{R}^2$ is called a weak critical point
of $E_V$ if $\dot{E}_V=0$ for all smooth variations $\dot{f}$ compactly
supported in $\reg(f) \cap \interior(M)$.
\end{definition}

\begin{proposition}
\label{eulerV}
An almost immersed map $f: M \to \mathbb{R}^2$ is a weak critical point of $E_V$
if and only if the function
\begin{equation}
\label{eqn:g}
g := (1+ V'(|f_z|^2)\,f_z
\end{equation}
defined on $\reg(f)$ is holomorphic. If this is the case, $g$ will always extend to a holomorphic function on the whole of $\interior(M)$.
\end{proposition}

\begin{proof}
Let us define the vector-valued 1-form
\begin{equation}
\label{sigma}
\sigma:=i\left( -V'\left(\left|f_z\right|^2\right)f_z dz + f_{\bar{z}}d\bar{z}\right).
\end{equation}
In the language of continuum mechanics \cite{marsden,sifakis} $\sigma$ can be described as the hodge-dual of the first Piola-Kirchhoff stress tensor.

Let $f$ and $\dot{f}$ be given as above. We compute the corresponding variation
of the energy:

\begin{align}
	\dot{E}_V(f) &= \left. \frac{d}{dt}\right|_{t=0}\frac{1}{2}\int_M \left( V\left(\left|f_z\right|^2\right) + \left|f_{\bar{z}}\right|^2\right)  \frac{i}{2} dz \wedge d\bar{z}\\ \nonumber
		&= \int_M \left( V'\left(\left|f_z\right|^2\right)\left\langle \dot{f}_z,f_z\right\rangle + \left\langle \dot{f}_{\bar{z}},f_{\bar{z}}\right\rangle\right) \frac{i}{2} dz \wedge d\bar{z}\\ \nonumber
		&= \frac{1}{2}Re \int_M \left( V'\left(\left|f_z\right|^2\right)\overline{\dot{f}_{z}}f_{z} + \overline{\dot{f}_{\bar{z}}}f_{\bar{z}}\right) i dz \wedge d\bar{z}\\ \nonumber
	&=\frac{1}{2} Re \int_M i\left( -V'\left(\left|f_z\right|^2\right)\overline{\dot{f}_{z}}d\bar{z}\wedge f_z dz + \overline{\dot{f}_{\bar{z}}}dz \wedge f_{\bar{z}}d\bar{z}\right)\\ \nonumber
 &= \frac{1}{2}Re\int_M \left( \overline{\dot{f}_{\bar{z}}}dz + \overline{\dot{f}_{z}}d\bar{z} \right) \wedge i\left( -V'\left(\left|f_z\right|^2\right)f_z dz + f_{\bar{z}}d\bar{z}\right)\\ \nonumber
		&= - \frac{1}{2} Re \int_M \overline{d\dot{f}}\wedge \sigma \\    \nonumber
		&= -\frac{1}{2} \int_M \langle d\dot{f} \wedge \sigma \rangle \\   \nonumber
		&=  \frac{1}{2} \int_M \langle \dot{f},d\sigma \rangle -  \frac{1}{2} \int_M d\langle \dot{f},\sigma \rangle. \label{comp}
\end{align}
$\sigma$ and $\dot{f}$ are smooth on $\reg(M)$ and $\dot{f}$ is compactly supported in $\reg(M) \cap \interior(M)$. Therefore, we can use Stokes theorem to obtain:
\begin{equation}
\dot{E}_V  = \frac{1}{2} \int_M \langle \dot f, d\sigma \rangle - \frac{1}{2}\int_{\partial M} \langle \dot f , \sigma \rangle = \frac{1}{2} \int_M \langle \dot f, d\sigma \rangle. \label{Edot1}
\end{equation}
As a consequence, $f$ is a weak critical point of $E_V$ if and only if $\sigma$ is a closed form on $\reg(f) \cap \interior(M)$. That is why, for a weak critical point $f$ also $(df - i\sigma) = (1+ V'(|f_z|^2)\,f_z dz = g\, dz$ is a closed 1-form and thus $g$ is a holomorphic function on $\reg(f) \cap \interior(M)$. On the other hand, $V'(x)\sqrt{x} \in O(1)$ and $f$ is Lipschitz and therefore $g$ is bounded and extends to a holomorphic function on the whole of $\interior(M)$.
\end{proof}

\begin{figure}[h]
\includegraphics[width=\textwidth]{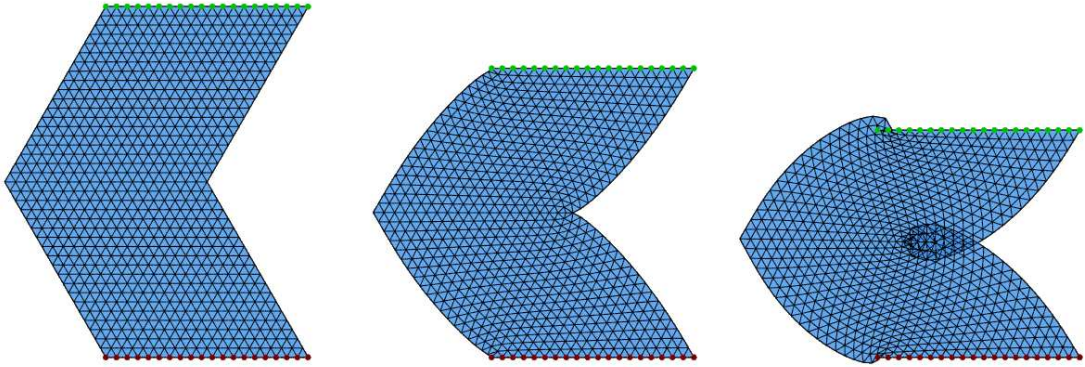}
\caption{A triangulated domain (left) is deformed elastically. During the deformation the points at the bottom are fixed while the points at the top are moved to prescribed positions. After a small elastic deformation the domain stays embedded (middle) and after a bigger deformation branch points appear (right). The deformations were computed numerically and visualized using Houdini.}
\label{branchpoint}
\end{figure}

\begin{definition}
We say that a map $f$ satisfies the {\em Euler-Lagrange equations} for $E_V$ if it is a weak critical point, i.e. if the function $g$ defined in \eqref{eqn:g} is holomorphic.
\end{definition}

\begin{definition}
An almost smooth map $f$ is called a {\em minimizer with fixed boundary} for $E_V$ if
\begin{equation*} \label{minimizer}
E_V(f) \leq E_V(\tilde{f})
\end{equation*}
for all almost smooth maps $\tilde{f}$ whose restriction to $\partial M$ is the
same as that of $f$. It is called a {\em strict minimizer} if (\ref{minimizer})
holds with strict inequality.
\end{definition}

\begin{definition}
An almost smooth map $f$ is called {\em locally unstable} for $E_V$ if there is
a point $p \in \interior(M)$ such that for every neighbourhood $U$ of $p$ there
are variations of $f$ supported in $U$ that bring down the energy $E_V$.
\end{definition}

\begin{proposition}
\label{prop minimizer}
An almost smooth map $f: M \to \mathbb{R}^2$ is a minimizer for $E_V$ with
fixed boundary if and only if it is a weak critical point for $E_V$ and
\begin{equation}
1+V'(|f_z|^2) \geq 0. \label{stability}
\end{equation}
If for a weak critical point (\ref{stability}) holds and the left
hand side does not vanish identically then $f$ is a strict minimizer. Weak critical points that are not
minimizers are locally unstable.
\end{proposition}

\begin{proof}
Suppose $f$ is an almost smooth solution of the Euler-Lagrange equation for
$E_V$ and (\ref{stability}) holds. Let $\tilde{f}: M \to \mathbb{R}^2$ be
another almost smooth map sharing with $f$ the same boundary values. Then
$\tilde f = f + h$ where $h$ is almost smooth and vanishes on $\partial M$.  
Due to the fact that $V$ is a smooth and strictly convex function we have
\begin{equation*}
V(x)  \geq V(y) + V'(y)(x-y), \quad \quad \quad \quad \forall x,y \in \left(0,\infty\right).
\end{equation*}
For the elastic energy this implies
\begin{align}
\label{E con}
E_V(f+h) &= \frac{1}{2}\int_M V(|f_z + h_z|^2) + |f_{\bar z} + h_{\bar z}|^2 \\
&\geq \frac{1}{2} \int_M V(|f_z|^2) + V'(|f_z|^2)(|f_z + h_z|^2 - |f_z|^2) + |f_{\bar
z} + h_{\bar z}|^2 \nonumber \\
&= E_V(f) + \frac{1}{2} \int_M V'(|f_z|^2)(2\langle f_z, h_z \rangle + |h_z|^2)
+ 2\langle f_{\bar z},h_{\bar z} \rangle + |h_{\bar z}|^2. \nonumber
\end{align} 
Due to the fact that $h$ vanishes on the boundary the oriented surface area of $h(M)$ is zero:
\begin{equation}
 \label{dh}
0 = \int_M \det(dh) = \int_M |h_z|^2 - |h_{\bar z}|^2.
\end{equation}
We insert this into (\ref{E con}) and obtain:
\begin{align}
\label{eq:minimizing-estimate}
&E_V(f+h) - E_V(f) \\
&\geq \frac{1}{2} \int_M (1+V'(|f_z|^2))|h_z|^2 +  \frac{1}{2} \int_M V'(|f_z|^2) 2\langle f_z, h_z \rangle
+ 2\langle f_{\bar z},h_{\bar z} \rangle. \nonumber
\end{align}
We now show that the second integral vanishes for almost smooth maps $h$ supported in $\interior(M)$:
\begin{align*}
&\int_M \left( V'(|f_z|^2) 2\langle f_z, h_z \rangle + 2\langle f_{\bar{z}},h_{\bar{z}} \rangle \right) \frac{i}{2} dz \wedge d\bar{z}\\ 
&=  \re \int_M \left(2V'(|f_z|^2) f_z \overline{h_z}  + 2 f_{\bar{z}} \overline{h_{\bar{z}}} \right) \frac{i}{2} dz \wedge d\bar{z} \\ 
&=  \re \int_M i\left( V'(|f_z|^2) f_z dz \wedge \overline{h_z} d\bar{z}  - f_{\bar z} d\bar{z} \wedge \overline{h_{\bar z}}  dz \right) \nonumber \\ 
&=  \re \int_M i\left( V'(|f_z|^2) f_z dz - f_{\bar z} d\bar{z}\right)  \wedge \left( \overline{h_z} d\bar{z} + \overline{h_{\bar z}}  dz \right) \\
&= \re \int_M \sigma \wedge \overline{dh} \\
&= \int_M \langle \sigma \wedge dh \rangle. \\
\end{align*}
Cutting out small disks around the points where $h$ and $\sigma$ are not smooth
we obtain a domain $M_0$ where we can apply Stokes theorem. Using $d\sigma = 0$
we get:
\begin{equation}
\label{m0}
\int_{M_0} \langle dh \wedge \sigma \rangle = \int_{\partial M_0} \langle h,
\sigma \rangle.
\end{equation}
Due to the fact that both $h$ and $\sigma$ are bounded the
boundaries of the disks do not contribute to (\ref{m0}) in the limit of
shrinking disks. Moreover, $h$ vanishes on $\partial M$ and we obtain $\int_M \langle dh
\wedge \sigma \rangle =0$ and therefore \eqref{eq:minimizing-estimate} becomes
\begin{equation}
E_V(f+h) - E_V(f) \geq \frac{1}{2} \int_M (1+V'(|f_z|^2))|h_z|^2. \nonumber
\end{equation}
This shows that $f$ is a minimizer with fixed boundary of the elastic energy provided that (\ref{stability})
holds. Moeover, $f$ will be a strict minimizer if $1+V'(|f_z|^2)$ does not vanish identically.

It remains to be proven that $f$ is locally unstable if there is a point $p \in
\reg(f) \cap \interior(M)$ with $1+V'(|f_z|^2) < 0$. Let $p$ be such a point,
$U \subset \reg(f) \cap \interior(M)$ any neighbourhood of $p$ and $h: \to
\mathbb{R}^2$ a smooth function compactly supported in $U$. We compute the
second derivative with respect to $t$ of $E_V(f+t h)$:
\begin{align}
\frac{d}{dt} E_V(f+t h) &=  \int_M V'(|f_z+t h_z|^2) \langle
f_z+t h_z, h_z \rangle + \langle f_{\bar z}+t h_{\bar z}, h_{\bar z} \rangle
\nonumber\\
\left. \frac{d^2}{dt^2}\right|_{t=0} E_V(f+t h) &=  \int_M 2 \, V''(|f_z|^2)
\langle f_z, h_z \rangle^2 + V'(|f_z|^2)|h_z|^2 + |h_{\bar z}|^2 \nonumber\\
&= 2 \int_M V''(|f_z|^2) \langle f_z, h_z \rangle^2 + \int_M (1+V'(|f_z|^2))|h_z|^2. 
\label{ustabel} 
\end{align}

Since the elastic energy is invariant under euclidean motions we can assume without loss of generality that $p=0$ and $f_z(p)$ is purely imaginary. Let $\phi:\, \mathbb{R} \to \mathbb{R}$ be a compactly supported even
function and define for $\epsilon > 0$ 
\begin{equation*}
h(z) = 1/\epsilon \, \phi(|z|/\epsilon) \, z.
\end{equation*}
Then $h_z$ is real valued and $\langle f_z(0), h_z(0) \rangle = 0$. It is now
easy to see that in the limit of small $\epsilon$ the first integral in
(\ref{ustabel}) goes to zero while the second one approaches a negative value.
The second variation of $E_V$ is therefore negative for small $\epsilon$.
\end{proof}

\section{Melting point solutions}
\label{chap melt}
For weak critical points $g=(1+ V'(|f_z|^2)\,f_z$ is a holomorphic function and $f_z$ has only isolated zeros. Therefore $(1+ V'(|f_z|^2)$ either has only isolated zeros (and thus $f$ is a strict minimizer or locally unstable) or it is identically zero. In the latter case $|f_z|$ will be constant.

\begin{definition}
An almost immersed map $f:M \mapsto \mathbb{R}^2$ is called melting point solution of $E_V$ if on the whole of $\reg(f)$ we have
\begin{equation}
\label{melt}
1+V'(|f_z|^2)=0.
\end{equation}
\end{definition}
For a melting point solution $|f_z|$ is a constant. In particular, for $E_\lambda$ the melting point condition (\ref{melt}) is equivalent to
\begin{equation*}
|f_z| =\frac{\lambda}{1+\lambda}.
\end{equation*}

From Proposition $\ref{eulerV}$ it follows that melting point solutions are weak critical points of $E_V$ and from Proposition $\ref{prop minimizer}$ we obtain that they are minimizers of $E_V$ but not necessarily strict ones. 

Using the language of Physics, if the material is compressed in such a way that $|f_z|$ somewhere falls below a critical lower bound, the material becomes unstable, it ``melts''.

Some melting point solutions can be obtained by a convergent sequence of strict minimizers of $E_V$ $\left(f_n \right)_{n\in \mathbb{N}}:M\mapsto \mathbb{R}^2$ whose limit 
satisfies $(\ref{melt})$. Not all melting point solutions arise this way:  
For strict minimizers $\mbox{arg}(f_z)$ is a harmonic function on $\mbox{reg}(f)$, so only those melting point solutions for which this also holds can be obtained as a limit of strict minimizers. We call such melting point solutions {\em borderline solutions}.

Not all melting point solutions are borderline solutions:
Let $M\subset \mathbb{C}$ be domain that does not containing the origin and define
\begin{align*}
f:M &\mapsto \mathbb{C} \\ 
f(z) &= \frac{\lambda}{i\left( 1+\lambda\right) \overline{z}}\,e^{iz\overline{z}} + \overline{k}(z).
\end{align*}
where $k:M \mapsto \mathbb{C}$ is a suitable holomorphic function defined as in (\ref{k}). Then
\begin{equation*}
f_z= \frac{\lambda}{1+\lambda}\,e^{iz\overline{z}}
\end{equation*}
and $f$ is a melting point solution for $E_\lambda$, but $\mbox{arg}(f_z)=z \overline{z}$ is not harmonic.

\section{Weierstrass representation}
\label{cp weierstrass}
The well-known Weierstrass representation describes conformal parametrizations of minimal surfaces
$f: M \to \mathbb{R}^3$ ($M$ a simply connected planar domain) in terms of
two holomorphic functions $g,h$ on $M$:
\begin{equation*}
 f(z)=\mbox{Re} \int h(1-g^2,i(1+g^2),2g).
\end{equation*}
Here we obtain a similar result for \textit{elastic maps}.

\begin{theorem}
\label{thmweierstrass1}
Let $M\subset \mathbb{C}$ be a simply connected domain and $h:M\rightarrow \mathbb{C}$ a holomorphic function with only finitely many zeros. Define
\begin{equation*}
H:=\int h\, dz \qquad \qquad  G:=\int h^2\,dz.
\end{equation*} 
Then there exists a meromorphic function $k:M\rightarrow \hat{\mathbb{C}}$ such that for $\lambda \geq 0$
\begin{align*}
f:M &\rightarrow \mathbb{C} \\
f&= \frac{1}{2}\left( G + \frac{2\lambda}{1+\lambda}\frac{H}{\overline{h}}\right)  + \overline{k}.
\end{align*}
is an almost immersed map and a strict minimizer of $E_\lambda$.
\end{theorem}

\begin{proof}
If $p_1,\ldots,p_n$ are the zeros of $h$ then
\begin{align*}
f:M \setminus\lbrace p_1,\ldots,p_n\rbrace &\to \mathbb{C} \\
f&=\frac{1}{2}\left( G + \frac{2\lambda}{1+\lambda}\frac{H}{\overline{h}}\right),
\end{align*}
is a well defined smooth function on $M \setminus\lbrace p_1,\ldots,p_n\rbrace$ but has singularities at $p_1,\ldots,p_n$.

The second term on the right-hand side of
\begin{equation}
\label{Hoverh}
\frac{H(z)}{\overline{h(z)}} = \frac{H(z)-H(p_j)}{\overline{h(z)}} + \frac{H(p_j)}{\overline{h(z)}},
\end{equation}
is antimeromorphic and has poles at the zeros of $h$.  Near $p_j$ we can express $\frac{H(p_j)}{\overline{h(z)}}$ as a Laurent series
\begin{equation*}
\frac{H(p_j)}{\overline{h(z)}} = \sum_{m=-l_j}^{\infty} c_{jm} (\bar{z}-p_j)^m.
\end{equation*}
Now we define a meromorphic function $k: M \to \hat{\CC}$ by
\begin{equation*}
\overline{k(z)}:= -\sum_{j=1}^n \sum_{m=-l_j}^{-1} c_{jm} (\bar{z}-p_j)^m,
\end{equation*}
and obtain that the restriction of $\frac{H(p_j)}{\overline{h(z)}}+\overline{k(z)}$ to a small neighborhood of $p_j$ is antiholomorphic.
The first term on the right-hand side of \ref{Hoverh}, tends to zero as $z$ goes to $p_j$, because $H(z)-H(p_j)$ has a zero of degree $l_j+1$ at $p_j$ if $\overline{h(z)}$ has one of degree $l_j$. 
Therefore, after adding $\overline{k}$, $f$ extents to a continuous map (still called $f$) on the whole of $M$. 
In order to see that $f$ is an almost immersed map it remains to show that the derivatives of $f$ are bounded on $M \setminus\lbrace p_1,\ldots,p_n\rbrace$.
$f_z$ is bounded because $h$ is holomorphic on $M$ and
\begin{align}
\label{fz}
f_z &= \frac{1}{2}\left( h^2+\frac{2 \lambda}{1+\lambda}\frac{h}{\overline{h}}\right)=\frac{1}{2}\left( h\overline{h}+\frac{2 \lambda}{1+\lambda}\right)\frac{h}{\overline{h}} \\ \nonumber
|f_z| &=\frac{1}{2}\left( h\overline{h}+\frac{2 \lambda}{1+\lambda}\right).
\end{align}
Now we consider $f_{\bar{z}}$ in a neighborhood $U_j$ of $p_j$
\begin{align*}
\frac{1+\lambda}{\lambda}f_{\bar{z}} = \left(\frac{H}{\overline{h}}\right)_{\bar{z}} + \overline{k}_{\bar{z}}
= \left(\frac{H - H(p_j)}{\overline{h}}\right)_{\bar{z}} + \left(\frac{H(p_j)}{\overline{h}} + \overline{k}\right)_{\bar{z}}.
\end{align*}
The second term is antiholomorphic and therefore bounded. In order to investigate the first term note that on $U_j$ there are nowhere vanishing holomorphic functions $a_j,b_j: U_j \to \CC$ such that
\begin{equation*}
\overline{h(z)} = (\overline{z-p_j})^{l_j} \,\overline{a_j(z)} \quad \textnormal{ and } \quad H(z)-H(p_j) = (z-p_j)^{l_j+1}\, b(z).
\end{equation*}
Then on $U_j$ we obtain
\begin{equation*}
\left| \left(\frac{H(z)-H(p_j)}{\overline{h(z)}}\right)_{\bar{z}} \right| = \left| \frac{\tilde{a}(z)(l_j\overline{a_j(z)} + (\bar{z}-p_j)\overline{a_j(z)}')}{\overline{a_j(z)}^2} \right|.
\end{equation*}
Hence both $f_z$ and $f_{\bar{z}}$ are bounded and therefore $f$ is an almost immersed map.

$f$ is a weak critical point of $E_\lambda$ if and only if away from finitely many points $g = \left(1+\lambda \right)f_z - \lambda \frac{f_z}{|f_z|} $ is a holomorphic function. This is indeed the case, by \ref{fz} we have
\begin{align}
\label{comp2}
g =\left(1+\lambda \right)f_z - \lambda \frac{f_z}{|f_z|} =\left(\frac{1+\lambda}{2}h\overline{h}+\lambda -\lambda \right)\frac{h}{\overline{h}}= \frac{1+\lambda}{2}h^2. 
\end{align}
In order for $f$ to be a strict minimizer we in addition must have
\begin{equation*}
1 + V'\left(\left|f_z\right|^2\right)\geq 0
\end{equation*} where the left side does not vanish identically.
For $E_\lambda$ we have $V(x)= \lambda \left(\sqrt{x}-1\right)^2$ and therefore we obtain
\begin{align*}
1 + V'\left(\left|f_z\right|^2\right) &= 1+\lambda\left(1 -\frac{1}{|f_z|}\right)\\ 
&= (1+\lambda)-\frac{\lambda}{\frac{1}{2}\left( |h|^2+\frac{2 \lambda}{1+\lambda}\right)} \\ 
&= \frac{|h|^2(1+\lambda)^2}{2\lambda+|h|^2(1+\lambda)}\geq 0.
\end{align*}
Since we assumed that $h$ does not vanish identically, neither does $1 + V'\left(\left|f_z\right|^2\right)$.
\end{proof}
\begin{figure}
\includegraphics[width=0.8\textwidth]{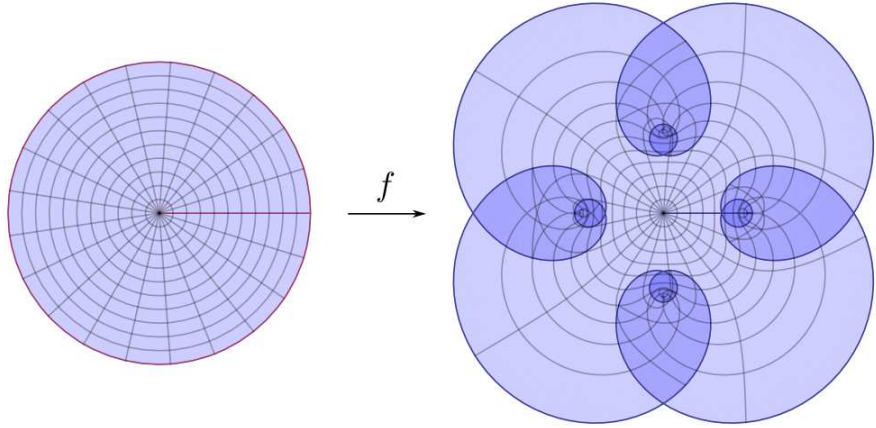}
\caption{The disk with center $(0,0)$ and radius 1.2 is deformed by the elastic map $f$ based on the holomorphic function $h(z) = z^4-1$. Since $h$ has zeros at $ 1,-1,i,-i$ we choose $k(z):= \frac{1}{10}\left( \frac{1}{z+1} + \frac{1}{z-1} + \frac{1}{z-i} +\frac{1}{z+i} \right)$. This compensates the corresponding poles of $ \frac{\int h \,dz}{\bar{h}} $. Note that $f$ has a branch point wherever $h$ has a zero. Figures 2-5 were made with Wolfram Mathematica.}
\end{figure}
Note that the minimizers $f$ that can be constructed based on the Weierstrass representation given in Theorem \ref{thmweierstrass1} are not completely general: By (\ref{comp2}) all zeros of the holomorphic function $g$ corresponding to such an $f$ have even order. Nevertheless, under this additional assumption the converse of Theorem \ref{thmweierstrass1} is also true: 

\begin{theorem}
\label{thmweierstrass2}
Let $M \subset \mathbb{C}$ be a simply connected domain and $f:M \to \mathbb{C}$ a weak critical point of $E_\lambda$. By Proposition \ref{eulerV}
\begin{equation}
\label{define-g}
g:=\left(1+\lambda \right)f_z - \lambda \frac{f_z}{|f_z|}
\end{equation}
extends to a holomorphic function on the whole of $M$. Assume that all zeros of $g$ have even order. Then there is a holomorphic function $h$ on $M$ and a meromorphic function $k$ on $M$ such that
\begin{equation*}
f= \frac{1}{2}\left( G + \frac{2\lambda}{1+\lambda}\frac{H}{\overline{h}}\right)  + \overline{k}
\end{equation*}
where $H=\int h$ and $G=\int h^2$.
\end{theorem}

\begin{proof} 
By our assumptions there is a holomorphic function $h:M\to\mathbb{C}$ such that
\begin{equation*}
\frac{1+\lambda}{2}h^2 =g.
\end{equation*}
Choose holomorphic functions $G,H$ on $M$ such that $H'=h$ and $G'=h^2$. Then one can check that the function $k$ defined away from the zeros of $g$ by
\begin{equation*}
\overline{k}:=f-\frac{1}{2}\left( G + \frac{2\lambda}{1+\lambda}\frac{H}{\overline{h}}\right)
\end{equation*}
is holomorphic and has only poles at the zeros of $g$.
\end{proof}

It is indeed possible to obtain a Weierstrass representation for general strict minimizers. The only complication is that in general the Weierstrass data live on the Riemann surface obtained as the double cover of the elastic domain branched over the odd-order zeros of $g$. The main information that will allow us to construct examples is contained in the proof of the following theorem:

\begin{theorem}
\label{thmweierstrass3}
Let $M \subset \mathbb{C}$ be a simply connected domain and $g:M \rightarrow \mathbb{C}$ a holomorphic function with only finitely many zeros. Then there exists an almost immersed map $f:M \to \mathbb{C}$ such that  (\ref{define-g}) holds.
\end{theorem}

\begin{proof}
Let $\hat{M}$ be the double cover of $M$ branched over the zeros of $g$ that have odd order. Denote by $z : \hat{M} \rightarrow M$ the projection. On $\hat{M}$ we have a well defined holomorphic function $h: \hat{M}\rightarrow \mathbb{C}$ such that 
\begin{equation*}
h^2 = \frac{2g\circ z}{1+\lambda}.
\end{equation*}
Let $\tau:\hat{M} \rightarrow \hat{M}$ be the involution that interchanges the two sheets of $\hat{M}$. Then
\begin{align*}
\label{tau}
\tau^2&=\textrm{id} \\
z \circ \tau &= z \\
h \circ \tau &= -h. 
\end{align*}
Since we assumed that $M$ is simply connected the first Betti number of $\hat{M}$ is
\begin{equation*}
\dim H_1(\hat{M}) = n-1.
\end{equation*}
In general there will not be any function $H:\hat{M}\to \mathbb{C}$ such that $dH= h\,dz$. However, by a theorem of Yukio Kusunoki and Yoshikazu Sainouchi \cite{Kusunoki}, there is a holomorphic 1-form $\eta$ on $\hat{M}$ with the same zeros as $h \, dz$ and periods
\begin{equation}
\int_{\alpha_i}\eta = \overline{\int_{\alpha_i}h\,dz}.
\end{equation}
The 1-form $h\,dz$ changes sign under $\tau$ and therefore also its periods. The same then holds for $\eta$ and therefore
\begin{equation}
\label{omega}
\omega:= \frac{\eta - \tau^*\eta}{2}
\end{equation}
has the same periods as $\eta$. The 1-form $h\,dz-\overline{\omega}$ then has no periods whatsoever and therefore is exact: There is a holomorphic function $H$ on $\hat{M}$ such that
\begin{equation}
\label{Hhat}
dH=h\,dz-\overline{\omega}.
\end{equation}
We have $\tau^*\omega=-\omega$ and therefore (after adding a constant to $H$) we can assume
\begin{equation*}
H\circ \tau = -H.
\end{equation*}
Then $\frac{H}{\overline{h}}$ is invariant under $\tau$ and we can define $F: M \rightarrow \mathbb{C}$ by
\begin{equation}
\label{H}
F \circ z:= \frac{H}{\overline{h}}.
\end{equation}
Since $\eta$ has the same zeros as $h$, (\ref{omega}) and (\ref{Hhat}) imply that at a degree $m$ zero $q\in \hat{M}$ of $h$ the function $H$ has a zero of degree $m+1$.

Away from finitely many zeros of $g$ the function $F$ is smooth and $F_z$ is non-zero. Moreover, our assumptions imply that the derivative of $F$ is bounded and therefore $F$ is an almost immersed map.

Given an arbitrary holomorphic function $k$ on $M$ we now can define an almost immersed map $f:M\rightarrow \mathbb{C}$ as 
\begin{equation*}
f:= \int \frac{g}{1+\lambda} dz + \frac{\lambda}{1+\lambda}F + \overline{k}.
\end{equation*}
We have 
\begin{equation*}
F_z=\frac{h}{\overline{h}} \circ z
\end{equation*}
and therefore
\begin{equation*}
f_z = \frac{g}{1+\lambda} + \frac{\lambda}{1+\lambda}\,\frac{h}{\overline{h}}.
\end{equation*}
Now a direct calculation shows (\ref{define-g}).
\end{proof}
 
The above proof leads to a practical method that allows us to find the minimizer $f$ that corresponds to a given $g$: We have to find a holomorphic differential $\omega$ satisfying the following properties:
\begin{align}
\label{cond1}
\tau^*\omega &= -\omega \\
\label{cond2}
\int_{\alpha_i} \omega &=\overline{\int_{\alpha_i} h \,dz}
\end{align}
for some homology basis $\lbrace\alpha_1,\ldots, \alpha_{n-1}\rbrace$ of $\hat{M}$.
The zeros of $\omega$ are not important because poles of $H$ can always be compensated by adding to $f$ a suitable anti-meromorphic function $\overline{k}$. In our examples we find $\omega$ based on the ansatz
\begin{equation*}
\omega := \sum_{i=1}^n x_i \,h^{2i-1} dz, \quad \quad \quad  x_i \in \mathbb{C}.
\end{equation*}
For such an $\omega$ condition (\ref{cond1}) is automatically satisfied because we only sum over odd powers of $h$. The coefficients $x_i$ have to be chosen such that (\ref{cond2}) holds. This amounts to a linear system $Bx=a$ where $a\in \mathbb{C}^n$ is defined as $a_j:= \overline{\int_{\alpha_j}h \,dz}$ and $b_{ij}:= \int_{\alpha_j} h^{2i-1} dz$. 
\begin{figure}[h!]
\includegraphics[width=\textwidth]{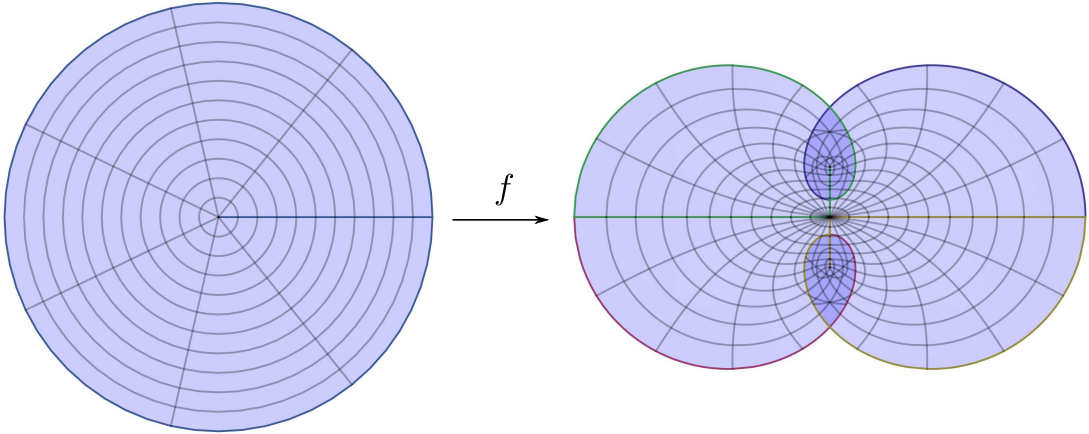} 
\caption{A circle with center $(0,0)$ and radius 1.4 and its image under the elastic deformation $f$ related to the holomorphic map $g(z)=z^2+1$. Since $g$ has two zeros of odd degree, $f$ has two branch points. }  
\label{weierstrass} 
\end{figure}
\begin{example}
Let us consider the case where $\lambda=1$, $M\subset \mathbb{C}$ is a simply connected domain containing $\pm i$ and $g(z):= z^2+1$. Since $g$ has two zeros of odd degree the first homology group of $\hat{M}$ has dimension one. One can show that the period $c:= \int_\alpha h\,dz$ of $h\,dz$ along any  non-trivial $\alpha \in H_1(\hat{M})$ is purly imaginary. So if we define
\begin{equation*}
\omega:= -\frac{c}{\overline{c}} \,h\,dz = h\,dz
\end{equation*}
then $\omega$ will satisfy (\ref{cond1}) and (\ref{cond2}). Now we define $F:M\rightarrow \mathbb{C}$ as in (\ref{H}) and obtain:
\begin{equation*}
F(z):= \frac{\mbox{Re} \left[\arcsinh(z)+z\sqrt{1+z^2}\right]}{2\sqrt{1+\overline{z}^2}}.
\end{equation*}
Choosing $k=0$ gives us the elastic map
\begin{equation*}
f(z):= \frac{1}{6}z^3 + \frac{1}{2}z + \frac{\mbox{Re} \left[\arcsinh(z)+z\sqrt{1+z^2}\right]}{4\sqrt{1+\overline{z}^2}}, 
\end{equation*}
that solves the differential equation $g = z^2+1 = (2 -\frac{1}{|f_z|})f_z.$ 
\end{example}

\subsection{Weierstrass representation of borderline solutions}
Not only strict minimizers but also borderline solutions admit a Weierstrass representation:
\begin{proposition}
\label{propweierstrass2}
Let $M\subset \mathbb{C}$ be a planar domain and $H,k:M\to\mathbb{C}$ two holomorphic functions and $\lambda>0$. Then

\begin{equation*}
f := \frac{\lambda}{1+\lambda}\frac{H}{\overline{h}} + \overline{k}
\end{equation*}

is a borderline solution for $E_{\lambda}$.
\end{proposition}

\begin{proof}
With $h:=H'$ we have
\begin{equation*}
f_z= \frac{\lambda}{1+\lambda}\,\frac{h}{\overline{h}}
\end{equation*}
and therefore $\mbox{arg}\,f_z$ is harmonic and $|f_z|$ is constant.
\end{proof}

\section{Deformations with free boundary}
\label{cp annu}
Up to now we only considered critical points $f:M\to \mathbb{R}^2$ of the elastic energy with respect to variations of $f$ with fixed boundary values. We call $f$ an elastic deformation with free boundary if $E(f)$ is critical under {\em all} variations of $f$, even if they move the boundary. According to (\eqref{Edot}) this amounts to saying that (in addition to solving the Euler-Lagrange equations in the interior) the restiction of the $\mathbb{R}^2$-valued 1-form $\sigma$ defined in \eqref{sigma} vanishes when applied to vectors tangent to the boundary:
\begin{equation*}
\sigma|_{T\partial M}=0.
\end{equation*}

Suppose $f$ is given by the Weierstrass representation in terms of two  holomorphic functions $h,k: M \rightarrow \mathbb{C}$. Then $f$ is elastic with free boundary if and only if for every local parametrization $\gamma:[a,b]\to\partial M$ of the boundary of $M$ we have
\begin{align}
0 &= \frac{1}{i}\sigma_\lambda(\gamma') = \lambda\left( 1-\frac{1}{|f_z|}\right) f_z \gamma' -f_{\bar{z}}\overline{\gamma'} \nonumber \\
&=   \lambda\left( \frac{1}{2}h\overline{h} - \frac{1}{1+\lambda} \right) \frac{h}{\overline{h}} \gamma' - \left( \overline{k'}-\frac{\lambda}{1+\lambda}\frac{H\overline{h'}}{\overline{h}^2}\right) \overline{\gamma'}. 
\label{freebd}
\end{align}

\subsection{Elastic strip with free boundary}
We want to find elastic equilibria of an annulus obtained by gluing two opposite sides of a rectangle. Here $M$ is not exactly a planar domain but at least a compact 2-dimensional manifold with boundary:
\begin{equation*}
M=[x_1,x_2]\times \mathbb{R}/_{2\pi\mathbb{Z}}.
\end{equation*}
In view of the periodic boundary conditions we make the following Ansatz for the Weierstrass data $h,k$: We choose $c>0, a=\alpha+i\beta \in \mathbb{C}, n\in\mathbb{N}$ and set
\begin{align*}
h(z) &= ce^{\frac{nz}{2}}\\ k(z)&= ae^{-nz}.
\end{align*}
\begin{figure}[h]
        \includegraphics[width=\textwidth]{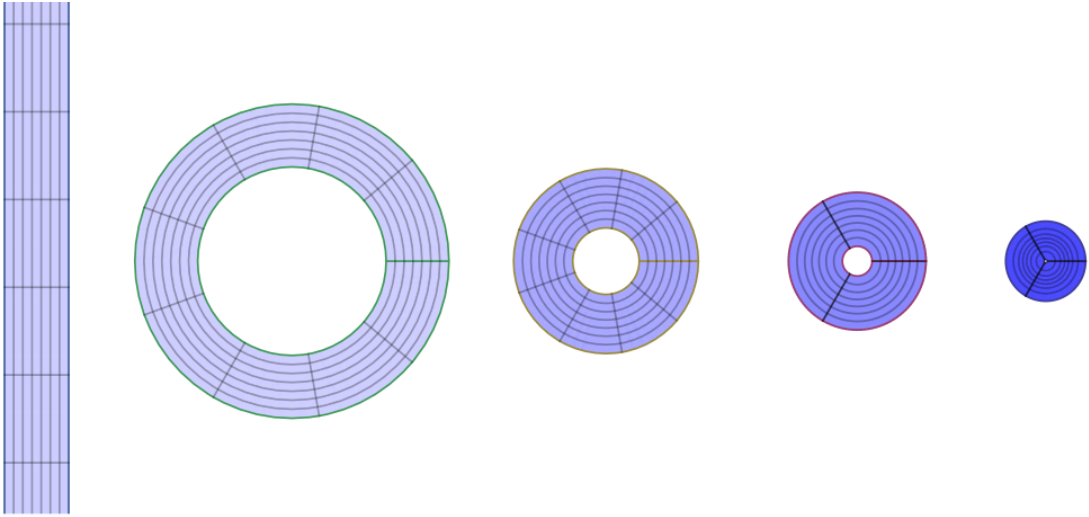}
        \caption{The rectangle $M:=\left[ \log\left(\frac{3}{4}\right) ,\log\left( \frac{5}{4}\right) \right] \times [0,2\pi]$ (only half of is shown in the picture) is bended elastically to an annulus with free boundary and winding number 1,2,3 or 6 respectively.}
        \label{picstripannuli}
\end{figure}
Suitable functions $H,G$ corresponding to this $h$ are
\begin{align*}
H(z)&=\frac{2c}{n}e^{\frac{nz}{2}} \\G(z)&= \frac{c^2}{n}e^{nz}
\end{align*}
and for $f$ we obtain
\begin{align*}
f(x,y) &= \frac{1}{2}\left( G + \frac{2\lambda}{1+\lambda}\frac{H}{\overline{h}}\right)  + \overline{k} \nonumber\\
&= \left( \frac{2\lambda}{(1+\lambda)n} + \frac{c^2}{2n}e^{nx} + (\alpha - i \beta)e^{-nx} \right)e^{iny}. 
\end{align*}
$f(M)$ is an annulus that winds around the origin $n$ times. By (\ref{freebd}) the boundary will be free if for every local parametrization $\gamma$ of $\partial M$
\begin{align*}
0 &=   \lambda\left( \frac{1}{2}h\overline{h} - \frac{1}{1+\lambda} \right) \frac{h}{\overline{h}} \gamma' - \left( \overline{k'}-\frac{\lambda}{1+\lambda}\frac{H\overline{h'}}{\overline{h}^2}\right) \overline{\gamma'} \nonumber \\
&= \lambda \left( \frac{c^2}{2}e^{nx}-\frac{1}{1+\lambda}\right) e^{i ny} \gamma'-\left[ (-\alpha + i \beta)ne^{n(-x+i y)} - \frac{\lambda}{1 + \lambda} \frac{e^{\frac{n(x+iy)}{2}}e^{\frac{n(x-iy)}{2}}}{e^{n(x-iy)}}\right] \overline{\gamma'} \nonumber \\
&= \lambda \left( \frac{c^2}{2}e^{nx}-\frac{1}{1+\lambda}\right) e^{i ny} \gamma'-\left[ (-\alpha + i \beta)ne^{-nx} - \frac{\lambda}{1 + \lambda} \right]e^{i ny} \overline{\gamma'}.
\end{align*}
The boundary components of $M$ are parallel to the y-axis, therefore $\gamma'=\pm i$ and we see that $\beta$ must be zero. Moreover, both for $x=x_1$ and $x=x_2$ we must have
\begin{align*}
0 &= \lambda \left( \frac{c^2}{2}e^{nx}-\frac{1}{1+\lambda}\right)-\alpha n e^{-nx} - \frac{\lambda}{1+\lambda} \nonumber \\
 &= \frac{\lambda c^2}{2}e^{nx} - \alpha ne^{-nx}-\frac{2\lambda}{1+\lambda}.
\end{align*} 
With $u:=e^{nx}$ this is a quadratic equation:
\begin{equation*}
0=u^2-\frac{4}{(1+\lambda)c^2}u-\frac{2\alpha n}{\lambda c^2} = \left(u-\frac{2}{(1+\lambda) c^2} \right)^2-\left( \left(\frac{2}{(1+\lambda)c^2} \right)^2+\frac{2\alpha n}{\lambda c^2} \right).  
\end{equation*}
We are interested in the case where this equation has two real roots, i.e. where there exists $b > 0$ such that:
\begin{equation*}
b^2=\left(\frac{2}{(1+\lambda)c^2} \right)^2+\frac{2\alpha n}{\lambda c^2}.
\end{equation*}
Then
\begin{align*}
x_1= \frac{1}{n} \log\left( \frac{2}{(1+\lambda)c^2}- b\right)\\
x_2= \frac{1}{n} \log\left( \frac{2}{(1+\lambda)c^2}+ b\right)
\end{align*}
and we have the following elastic annulus with free boundary: 
\begin{align*}
f:\left[\frac{1}{n} \log\left( \frac{2}{(1+\lambda)c^2} - b\right) ,\frac{1}{n} \log\left( \frac{2}{(1+\lambda)c^2} + b\right) \right] \times \mathbb{R} \setminus 2\pi \rightarrow \mathbb{R}^2 \nonumber, \\
f(x,y)= \left( \frac{2\lambda}{(1+\lambda)n} + \frac{c^2}{2n}e^{nx} + \left( \frac{\lambda c^2 b^2}{2n}-\frac{2\lambda}{(1+\lambda)^2nc^2}\right) e^{-nx} \right)e^{iny}. 
\end{align*}

\subsection{Elastic deformations with free boundary of a standard annulus}
The case of a standard annulus
\begin{equation*}
M=\{z\in \mathbb{C}\,|\, r_1\leq |z| \leq r_2\}.
\end{equation*}
yields other explicit elastic equilibria with free boundary. For the Weierstrass data $h,k$ we use the ansatz
\begin{align*}
h(z) &=c \,z^n\\ k(z) &= a \,z^{-2n-1}
\end{align*}
where $c>0$, $a \in \mathbb{C}$ and $2n \in \mathbb{N}$. We choose the integration constants in the corresponding functions $H,G$ as
\begin{align*}
H(z) &=\frac{c\,z^{n+1}}{n+1} \\ G(z) &= c^2 \, \frac{z^{2n+1}}{2n+1}
\end{align*}
and obtain
\begin{align}
f(z) &= \frac{1}{2}\left( G + \frac{2\lambda}{1+\lambda}\frac{H}{\overline{h}}\right) + \overline{k} \nonumber \\
&= \frac{c^2}{4n+2}z^{2n+1} + \frac{\lambda}{(1+\lambda)(n+1)}\frac{z^{n+1}}{\bar{z}^n} + \overline{a} \bar{z}^{-2n-1} \nonumber \\
&= \frac{z^{2n+1}}{|z|^{2n+1}}\left(\frac{c^2}{4n+2} |z|^{2n+1} + \frac{\lambda}{(1+\lambda)(n+1)}|z| + \overline{a}|z|^{-2n-1}\right). 
\label{f annu}
\end{align}
$f(M)$ is an annulus that winds $2n+1$ times around the origin. The boundary curves of $M$ are parametrized by $\gamma_{1/2}(t)=r_{1/2}e^{it}$. With (\ref{freebd}) the boundary of $f(M)$ is free if:
\begin{align*}
0 &=   \lambda\left( \frac{1}{2}h\overline{h} - \frac{1}{1+\lambda} \right) \frac{h}{\overline{h}} \gamma' - \left( \overline{k'}-\frac{\lambda}{1+\lambda}\frac{H\overline{h'}}{\overline{h}^2}\right) \overline{\gamma'} \\
&= \lambda\left( \frac{c^2}{2}|z|^{2n} - \frac{1}{1+\lambda} \right) \frac{z^n}{\bar{z}^n} \gamma'  - \left( -\overline{a}(2n+1)\bar{z}^{-2n-2}-\frac{\lambda}{1+\lambda}\frac{n}{n+1}\frac{z^{n+1}\bar{z}^{n-1}}{\bar{z}^{2n}}\right) \overline{\gamma'}.
\end{align*} 
Again we see that $a$ must be a real number $\alpha$ and
\begin{align}
0 &= \lambda\left( \frac{c^2}{2}r^{2n} - \frac{1}{1+\lambda} \right) -\alpha (2n+1)r^{-2n-2} -\frac{\lambda}{1+\lambda}\frac{n}{n+1} \nonumber \\
 0 &= r^{4n+2} -\frac{4n+2}{(1+\lambda)(n+1)c^2}r^{2n+2}-\frac{\alpha(4n+2)}{c^2\lambda} \label{r annulus}.
\end{align}
For all pairs of radii $0<r_1<r_2$ there always exists $\alpha \in \mathbb{R}$ and $c>0$ such that (\ref{r annulus}) is satisfied for $r=r_1$ and $r=r_2$. This can be seen by using the substitution $u:=\frac{1}{c^2}$ and $v:=\frac{\alpha}{c^2}$. The real numbers $u$ and $v$ have to solve the following linear system:
\begin{align}
r_2^{4n+2} &= \frac{4n+2}{(1+\lambda)(n+1)}r_2^{2n+2}u+\frac{4n+2}{\lambda}v, \nonumber \\
r_1^{4n+2} &= \frac{4n+2}{(1+\lambda)(n+1)}r_1^{2n+2}u+\frac{4n+2}{\lambda}v.
\label{lgs}
\end{align}
This system has a unique solution $(u,v)$ with $u\neq 0$ for all $0<r_1<r_2$ because the corresponding determinant is not zero:
\begin{equation}
(r_2^{2n+2}-r_1^{2n+2})\frac{4n+2}{(1+\lambda)(n+1)}\frac{4n+2}{\lambda} = (r_2^{2n+2}-r_1^{2n+2})\frac{(4n+2)^2}{(1+\lambda)(n+1)\lambda}\neq 0.
\end{equation}
Solving (\ref{lgs}) we obtain for $c$ and $\alpha$:
\begin{align}
c &= \sqrt{\frac{(4n+2)\left( r_2^{2n+2}-r_1^{2n+2}\right) }{(1+n)(1+\lambda)\left( r_2^{4n+2}-r_1^{4n+2}\right) }}, \nonumber\\
\alpha &= \frac{\lambda\, r_2^{2n+2}r_1^{2n+2} \left( r_2^{2n}-r_1^{2n}\right)}{(1+n)(1+\lambda)\left( r_1^{4n+2}-r_2^{4n+2}\right)}.
\label{c and alpha}
\end{align}

\begin{figure}[h]
        \includegraphics[width=\textwidth]{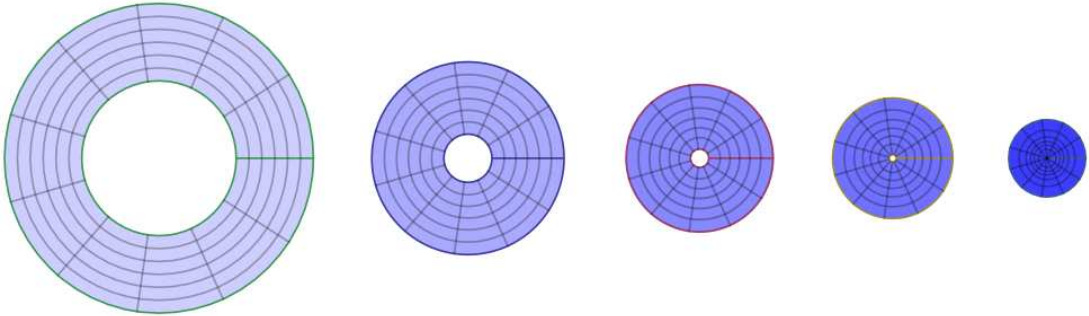}
        \caption{The annulus with radii $r_1=1$ and $r_2=2$ (left) is coiled up by an elastic deformation with free boundary. Its winding number is 2, 3, 4 and 7 respectively, so in the four pictures on the right the blue area is covered multiple times. For $n=7$ the hole of the wound up annulus becomes very small.}
        \label{picannuli}
   \end{figure}

\section{Uniqueness under free boundary conditions}
In this section we return to the general elastic energies $E_V$. We will prove that in the absence of boundary conditions and on a simply connected planar domain the only stable elastic maps are orientation preserving euclidean motions.

This result in fact also holds for ``planar domains with self-overlap'': Let $M$ be a compact connected and simply connected 2-manifold with boundary and $z: M \to \mathbb{R}^2 = \mathbb{C}$ an immersion. As in (\ref{splitting df}), the differential $df$ of any map $f: M \to \mathbb{C}$ can be uniquely decomposed as:
\begin{equation*}
df = f_z dz + f_{\bar z} d{\bar z}.
\end{equation*}
Then as in (\ref{EV}) we define
\begin{equation}
E_V(f) = \frac{1}{2} \int_M V(|f_{z}|^2) + |f_{\bar{z}}|^2. \label{EV2}
\end{equation}
Here the integral is taken with respect to the volume form $\frac{i}{2}dz \wedge d{\bar z}$ induced on $M$ by the immersion $z$.
\begin{theorem}
Let $f: M \to \mathbb{R}^2$ be an orientation-preserving immersion that is a critical point for $E_V$
with respect to all variations of $f$. Suppose that on all of $M$ we have
\begin{equation}
\label{V'}
1+V'(|f_z|^2) > 0.
\end{equation}
Then $f$ is an orientation preserving euclidean motion.
\end{theorem}

\begin{proof}
By Proposition \ref{eulerV} the function
\begin{equation}
g = (1+V'(|f_{z}|^2)) f_z \label{g}
\end{equation}
is holomorphic with respect to $z$. 

Since $f$ is an orientation-preserving immersion, $f_z$ has no zeros. Using this and (\ref{V'}) we see that also $g$ has no zeros and since $M$ is simply connected we can define $\log g$ globally on $M$. Moreover, there is a function $\dot{f}: M \to \mathbb{C}$ such that:
\begin{align}
\dot{f}_z &= -g \log g, \nonumber \\
\dot{f}_{\bar z} &= 0.
\end{align}
If we use $\dot{f}$ as an infinitesimal variation of $f$, the corresponding variation of the energy is
\begin{align}
\label{Edot}
\dot{E}_V
&= \int_M V'(|f_{z}|^2)\langle \dot{f}_z, f_z \rangle + \langle \dot{f}_{\bar{z}},f_{\bar{z}} \rangle \\
&= \int_M -V'(|f_{z}|^2)\langle(1+V'(|f_{z}|^2)) \log ((1+V'(|f_{z}|^2))f_{z})f_{z},f_z\rangle \nonumber \\
&= \int_M -V'(|f_{z}|^2)(1+V'(|f_{z}|^2)) \re \left[ \log ((1+V'(|f_{z}|^2))f_{z}) \right]  |f_{z}|^2 \nonumber \\
&= \int_M -V'(|f_{z}|^2)(1+V'(|f_{z}|^2)) \log ((1+V'(|f_{z}|^2))|f_{z}|) |f_{z}|^2. \nonumber \label{Edot}
\end{align}

Due to the fact that $V:(0,\infty)\rightarrow \mathbb{R}$ is smooth, strictly convex and takes its only minimum at $x=1$ we have
\begin{align*}
 V'(x) &< 0 \quad \text{for} \quad x \in (0,1) \\
 V'(x) &> 0 \quad \text{for} \quad x \in (1,\infty).
\end{align*}
Thus for $0<u<1$ we have:
\begin{align*}
V'(u) &< 0 \\
(1+V'(u))u &< 1 \\
\log((1+V'(u))u) &< 0.
\end{align*}
Therefore, the integrand of (\ref{Edot}) is negative at all points of $M$ where $|f_z| < 1$. By a similar argument, the same is true for $|f_z| > 1$. This means that everywhere we have $|f_z|=1$, because otherwise the variation of $E_V$ would be negative. This would contradict our assumption that $f$ is a critical point of $E_V$.

The holomorphicity of $g$ in (\ref{g}) then implies that also $f_z$ is a holomorphic function. Because of $|f_z| =1$ we conclude that $f_z$ is constant. Moreover, $|f_z|=1$ and $V'(1)=0$ the 1-form $\sigma$ defined in (\ref{sigma}) now has the form
\begin{equation}
\sigma = - i f_{\bar{z}} d\bar{z}.
\end{equation}
We know $f_{z\bar{z}} = 0$ and therefore $f_{\bar{z}}$ is an antiholomorphic function. By the left equality in (\ref{Edot1}) criticality of $f$ with respect to all variations implies that the $\sigma$ has to vanish on vectors tangent to $\partial M$. So the antiholomorphic function $f_{\bar{z}}$ vanishes on $\partial M$ and thus has to vanish identically. This means that $f$ is a holomorphic map whose derivative $f_z$ is constant and has unit norm. In other words, $f$ is an orientation preserving euclidean motion.
\end{proof}
\begin{figure}[h]
    \centering
     \includegraphics[width=\textwidth]{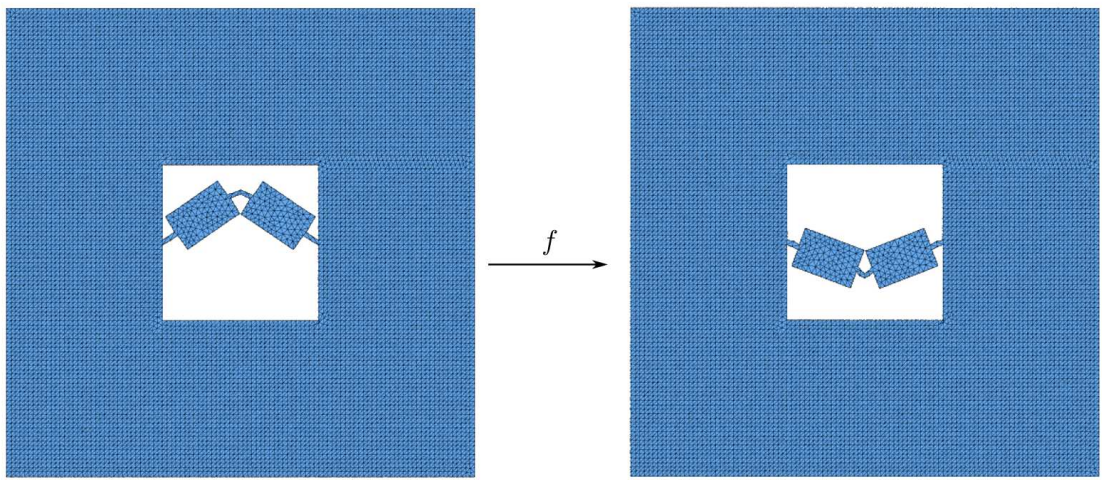}
    \caption{A non-trivial elastic deformation with free boundary of a domain $M$ that is not simply connected}\label{gegenbsp}
\end{figure}

Note that in dimensions greater than 2 there are counterexamples to the above theorem, and also the condition that $M$ is simply connected cannot be dropped:
\begin{itemize}
\item  A thickened half-sphere in three dimensions can be turned inside out to yield a
non-trivial stable equilibrium.
\item In \figref{picannuli} an annulus is elastically deformed to another annulus with higher winding number and free boundary. Here the image homotopy class gets changed to obtain an non-trivial stable equilibrium.
\item An example for an elastic deformation with boundary that preserves the homotopy class was constructed numerically with Houdini and is shown in \figref{gegenbsp}.
\end{itemize}

\bibliographystyle{abbrv}
\bibliography{dhlb}{}

\end{document}